\documentclass[11pt]{amsart}

\usepackage{amsthm, amsfonts, amssymb, amscd, rotating, amsmath}
\usepackage[pagebackref,colorlinks]{hyperref}
\usepackage{tikz-cd}
\usepackage[most]{tcolorbox}

\theoremstyle{definition}
\newtheorem{ntn}{Notation}[section]
\newtheorem{dfn}[ntn]{Definition}
\theoremstyle{plain}
\newtheorem{lem}[ntn]{Lemma}
\newtheorem{prp}[ntn]{Proposition}
\newtheorem{thm}[ntn]{Theorem}
\newtheorem{cor}[ntn]{Corollary}

\theoremstyle{definition}

\newtheorem{rem}[ntn]{Remark}
\newtheorem{exa}[ntn]{Example}

\numberwithin{equation}{section}

\newcommand{\N}{\mathbb{N}}
\newcommand{\z}{\mathbb{Z}}
\newcommand{\q}{\mathbb{Q}}

\newcommand{\C}{\mathbb{C}}
\newcommand{\F}{\mathbb{F}}

\newcommand{\ppp}{\mathfrak{p}}
\newcommand{\qqq}{\mathfrak{q}}
\newcommand{\mmm}{\mathfrak{m}}

\newcommand{\OO}{\mathcal{O}}

\renewcommand{\SS}{\mathcal{S}}

\renewcommand{\aa}{{A^\times}}

\newcommand{\se}{\subseteq}
\newcommand{\arr}{\rightarrow}

\newcommand{\two}{\twoheadrightarrow}

\newcommand{\GL}{{\rm GL}}

\newcommand{\PSL}{{\rm PSL}}
\newcommand{\SL}{{\rm SL}}
\newcommand{\GE}{{\rm GE}}
\newcommand{\Ee}{{\rm E}}

\renewcommand{\char}{{\rm char}}

\newcommand{\ab}{{\rm ab}}

\newcommand{\Spec}{{\rm Spec}}

\newcommand {\mtxx}[4]
{\left(\!
\begin{array}{cc}
\!\!#1 & \!\!#2 \\
\!\!#3 & \!\!#4
\end{array}\!\!
\right)}

\newtheoremstyle{athm}
{}
{}
{\itshape}
{}
{\scshape}
{}
{.5em}
{\thmnote{#3}}
\theoremstyle{athm}

\begin{document}

\title[The abelianization of $\text{SL}_2$]
{Abelianization of $\text{SL}_2$ over Dedekind domains of arithmetic type}
\author{Behrooz Mirzaii, Bruno R. Ramos, Thiago Verissimo}
\address{\sf Instituto de Ci\^encias Matem\'aticas e de Computa\c{c}\~ao (ICMC),
Universidade de S\~ao Paulo, S\~ao Carlos, Brasil}
\email{bmirzaii@icmc.usp.br}
\email{brramos@usp.br}
\email{thiagovlg@usp.br}

\begin{abstract}
We determine the exact group structure of the abelianization of $\text{SL}_2(A)$, in which $A$ is 
a Dedekind domain of arithmetic type with infinitely many units. In particular, our results show 
that $\text{SL}_2(A)^\text{ab}$ is finite, with exponent dividing $12$ when $\text{char}(A)=0$, 
and dividing $6$ when $\text{char}(A)>0$. As illustrative examples, we compute $\text{SL}_2(A)^\text{ab}$ 
explicitly for instances where $A$ is the ring of integers of a real quadratic field or a cyclotomic 
extension. \\

\end{abstract}

\maketitle

%%%%%%%%%%%%%%%%%%%%%%%%%%%%%%%%%%%%%%%%%%%%%%%%%%%%%%%%%%%%%%%%%%%%%%%%%%%%%%%%%%
\section*{Introduction}
%%%%%%%%%%%%%%%%%%%%%%%%%%%%%%%%%%%%%%%%%%%%%%%%%%%%%%%%%%%%%%%%%%%%%%%%%%%%%%%%%%

We investigate the first integral homology group of $\SL_2(A)$, where $A$ is a Dedekind 
domain of arithmetic type. By definition, a Dedekind domain of arithmetic type is the ring of 
$S$-integers in a global field $K$, where $S$ is a set of primes containing all infinite 
primes (see Section~\ref{sec1}).

Our primary contribution is the explicit determination of the group structure of the first 
integral homology of $\SL_2(A)$, where $A$ is a Dedekind domain of arithmetic type with 
infinitely many units. Since the first integral homology group of a group $G$ is naturally 
isomorphic to its abelianization $G^\ab :=G/[G,G]$, our results yield the precise structure 
of $\SL_2(A)^\text{ab}$. 

Let $A$ be a Dedekind domain of arithmetic type with characteristic zero and fraction field 
$K$. Denote by $\OO_K$ the ring of integers of $K$. We now consider the following sets of 
primes in $\OO_K$:
\[
S_2=\{\ppp\in \Spec(\OO_K): \ppp\mid 2, e_\ppp=1,  [\OO_K/\ppp: \F_2]=1\},
\]
\[
S_2'=\{\ppp\in \Spec(\OO_K): \ppp\mid 2, e_\ppp>1, [\OO_K/\ppp: \F_2]=1\},
\]
\[
\!\!\!\!\!\!\!\!\!\!\!\!\!\!\!\!\!\!
S_3=\{\ppp\in \Spec(\OO_K): \ppp \mid 3, [\OO_K/\ppp: \F_3]=1\},
\]
where $e_\ppp$ denotes the ramification index of the prime $\ppp$ and $[\OO_K/\ppp: \F_p]$ 
is the inertia degree of $\ppp$ over $p$ (see Section~\ref{sec1}).

If $A=\OO_{K,S}$ and $S$ contains at least two primes, including all infinite primes, then 
we show that
\[
\SL_2(A)^\ab \simeq
%(\z/4)^{|S_2\backslash S|} \oplus (\z/2)^{2|S_2'\backslash S|} \oplus (\z/3)^{|S_3\backslash S|}.
\bigoplus_{\ppp\in S_2\backslash S} \z/4 \oplus 
\bigoplus_{\ppp\in S_2'\backslash S} (\z/2\oplus \z/2) \oplus
\bigoplus_{\ppp\in S_3\backslash S} \z/3.
\]
In particular, $\SL_2(A)^\ab$ is of exponent dividing $12$ (see Theorem \ref{Main}).

Let us now consider the case where $A=\OO_{K,S}$ is of positive characteristic and 
$[K:\F_q(t)]<\infty$. If $q=2$, define
\[
S_2''=\{\ppp\in\Spec(\OO_K): \ppp \mid t-a\ \text{, for some $a\in \F_2$, and}\ [\OO_K/\ppp:\F_2]=1\}.
\]
Similarly, if $q=3$, define
\[
S_3''=\{\ppp\in\Spec(\OO_K): \ppp \mid t-a\ \text{, for some $a\in \F_3$, and}\ [\OO_K/\ppp:\F_3]=1\}.
\]
We show that if  $S$ contains at least two primes, including all infinite primes, then
\[
\SL_2(A)^\ab \simeq 
%\begin{cases}
%(\z/2)^{2|S_2''\backslash S|} &  \text{if $q=2$}\\
%(\z/3)^{|S_3''\backslash S|}  & \text{if $q=3$.}\\
%0 & \text{if $q>3$}
%\end{cases}
\begin{cases}
\bigoplus_{\ppp\in S_2''\backslash S}(\z/2\oplus \z/2) &  \text{if $q=2$}\\
\bigoplus_{\ppp\in S_3''\backslash S} \z/3  & \text{if $q=3$.}\\
0 & \text{if $q>3$}
\end{cases}
\]
In particular, $\SL_2(A)^\ab$ is of exponent dividing $6$ (see Theorem \ref{main2}). 

Using the above structural result obtained for the abelianization of $\SL_2$ of Dedekind domain 
of arithmetic type $A=\OO_{K,S}$ of characteristic zero, we examine two 
specific cases: when $K$ is a quadratic number field, and when $K$ is a Galois extension of $\q$.
We show that for a square-free positive integer $d$, if $\OO_d$ denotes the ring of integers 
of $K=\q(\sqrt{d})$, then (see Theorem \ref{d>0})
\[
\SL_2(\mathcal{O}_d)^\ab\simeq \begin{cases}
0                &\text{if $d\equiv 5 \pmod  {24}$}\\
\z/12\oplus\z/12 &\text{if $d\equiv 1 \pmod  {24}$}\\
\z/12\oplus \z/4 &\text{if $d\equiv 9 \pmod  {24}$}\\
\z/3 \oplus \z/3 &\text{if $d\equiv 13 \!\!\!\pmod  {24}$}\\
\z/3             &\text{if $d\equiv 21 \!\!\!\pmod  {24}$}\\
\z/4 \oplus \z/4 &\text{if $d\equiv 17 \!\!\!\pmod  {24}$}\\
\z/2 \oplus \z/2 &\text{if $d\equiv 2, 11, 14, 20, 23 \!\!\!\pmod  {24}$}\\
\z/6 \oplus \z/6 &\text{if $d\equiv 4,7,10,19, 22 \pmod  {24}$}\\
\z/6 \oplus \z/2 &\text{otherwise.}
\end{cases}
\]

When considering $\OO_{-d}$, the ring of integers of $K=\q(\sqrt{-d})$, for a square-free negative 
integer $-d$, our earlier results do not directly apply, since $\OO_{-d}$ has only finitely many 
units. The group structure of $\SL_2(\mathcal{O}_{-d})$ is significantly more intricate and has 
been explicitly computed only in specific cases (see, e.g., \cite{cohn1968}, \cite{swan1971}, 
\cite{rahm2013}, \cite{rahm-2013}). However, if we instead consider the localization $\OO_{-d, S}$, 
where $S$ is a set of primes containing at least two elements, including the infinite prime, 
we establish a result analogous to the positive characteristic case discussed earlier 
(see Theorem~\ref{d<0}).

Finally, we examine the abelianization $\SL_2(A)^\ab$, where $A$ is the ring of integers of a Galois 
extension $K/\q$ with $[K:\q]>2$ (see Theorem \ref{Galois}). In particular, we show that, for cyclotomic 
extension $\q(\zeta_N)/\q$, the group $\SL_2(\mathbb{Z}[\zeta_N])^\ab$ satisfies
\[
\SL_2(\mathbb{Z}[\zeta_N])^\ab \simeq \begin{cases}
\z/12     &  \text{if $N=1,2$}\\
\z/2\oplus \z/2  &  \text{if $N=2^k$, $k\geq 2$}\\
\z/3      & \text{if $N=2^{k}3^m$, $k\in \{0,1\}$, $m>0$}\\
0 & \text{otherwise}
\end{cases}
\]
(see Proposition \ref{cyclotomic}).

After completing this article, we realized that, in \cite[Corollary 2]{BS2013}, the order of $\SL_2(A)^\ab$ is 
given when $A$ is the ring of integers of a number field with at least one real embedding. However, no group 
structure is provided. Our main results (Theorems \ref{Main} and \ref{main2}) thus generalize this to a 
significantly broader class of Dedekind domains of arithmetic type and provide an explicit description 
of the group structure in question.

~\\
{\bf Acknowledgments.} 
We'd like to thank Carl Nyberg-Brodda for his constructive feedback on the group $\SL_2(\z[1/n])^\ab$, 
which we discuss in Example~\ref{Z1-n}. This research was made possible, in part, by the support 
provided to the second and third authors through CAPES (Coordena\c{c}\~ao de Aperfei\c{c}oamento 
de Pessoal de N\'ivel Superior) Ph.D. and M.Sc. fellowships (grant numbers 88887.983475/2024-00 
and 88887.955905/2024-00).

%%%%%%%%%%%%%%%%%%%%%%%%%%%%%%%%%%%%%%%%%%%%%%%%%%%%%%%%%%%%%%%%%%%%%%%%%%%%%%%%%%
\section{Dedekind domains of arithmetic type}\label{sec1}
%%%%%%%%%%%%%%%%%%%%%%%%%%%%%%%%%%%%%%%%%%%%%%%%%%%%%%%%%%%%%%%%%%%%%%%%%%%%%%%%%%

A global field $K$ is defined as either a finite extension of the rational numbers $\q$, or a 
finite extension of a rational function field $\F_q(t)$, for some finite field $\F_q$.
In the latter case, we may assume that $K$ is separable over $\F_q(t)$
\cite[page 160]{rv1998}.

The ring of algebraic integers of $K$, denoted by $\OO_K$, is the subring of all elements of 
$K$ that are integral over $\z$ (if $K$ has characteristic zero) or over $\F_q(t)$
(if $K$ has positive characteristic). Notably, $\OO_K$ is a Dedekind domain.

If $K$ has characteristic zero ($\char(K)=0$), any embedding into $\C$ defines an archimedean 
absolute value on $K$. However, if $K$ has positive characteristic ($\char(K)>0$) and is a 
finite extension of $\F_q(t)$, we consider the absolute value on $\F_q(t)$ given by 
$|f(t)/g(t)|_\infty:=q^{\deg(f)-\deg(g)}$.
This absolute value extends to a finite number of non-archimedean absolute values on $K$ 
\cite[Proposition 4.31]{rv1998}. In both scenarios, these specific absolute values 
are known as the {\it infinite primes} (or infinite places) of $K$ \cite[page 164]{rv1998}.

Any non-trivial prime ideal of $\OO_K$ directly corresponds to a non-trivial non-archimedean 
absolute value on $K$. These specific absolute values are what we refer to as the {\it finite primes} 
(or finite places) of $K$.

A {\it prime} of $K$ is simply either an infinite prime or a finite prime. Notably, each prime ideal 
$\ppp$ in $\OO_K$ defines a discrete valuation $v_\ppp: K \to \z\cup\{\infty\}$ \cite[\S~6.1]{keune2023}.

Let $S$ be a finite set of primes of $K$ that includes all infinite primes. The ring of $S$-integers 
in $K$, denoted by $\OO_{K,S}$, is defined as:
\[
\OO_{K,S}=\{x\in K: v_\ppp(x)\geq 0, \ \text{for all}\ \ppp\notin S\}.
\]
It is worth noting that if the set $S$ includes only the infinite primes, then the ring of $S$-integers, 
$\OO_{K,S}$, is simply the ring of integers of $K$, $\OO_K$ \cite[Proposition~4.35]{rv1998}.

It is a well-known result that $\OO_{K,S}$ is a Dedekind domain (see \cite[Theorem~6.24]{keune2023}).
In fact, $\OO_{K,S}\simeq {\mathcal S}^{-1}\OO_K$, where 
$\mathcal{S}=\OO_K\backslash \bigcup_{\ppp\in\Spec(\OO_K)\backslash S}\ppp$ 
\cite[Exercise~14, Chap.~6]{keune2023}. The definition provided below is adapted from 
\cite[page~83]{bms1967}.

\begin{dfn}
A ring $A$ is a Dedekind domain of arithmetic type if $A=\OO_{K,S}$, for some global 
field $K$ and for some finite non-empty set $S$ of primes of $K$ that includes all of its 
infinite primes.
\end{dfn} 

It is well known that if $I$ is a nontrivial ideal of a Dedekind domain $A$, then every ideal 
of the quotient ring $A/I$ is principal \cite[Exercise 7, Chap. 9]{am1969}. The following fact 
is also well known, but we include a proof for completeness.

\begin{lem}\label{PFR}
Let $A$ be a Dedekind domain of arithmetic type. If $I$ is a nontrivial ideal of $A$, 
then the quotient ring $A/I$ is finite.
\end{lem} 
\begin{proof}
Let $A=\OO_{K,S}$ and let $I=\ppp_1^{\alpha_1}\cdots \ppp_r^{\alpha_r}$ be the prime 
decomposition of $I$ in $A$. By the Chinese Remainder Theorem, 
$A/I\simeq A/\ppp_1^{\alpha_1}\times \cdots \times A/\ppp_r^{\alpha_r}$. 
Hence, to prove the lemma, it suffices to show that for every nontrivial prime ideal
$\ppp$ of $A$ and every positive integer $\alpha$, 
$A/\ppp^\alpha$ is finite. We proceed by induction on $\alpha$. Let $\alpha=1$ and let 
$\qqq=\OO_K\cap \ppp$. Then $A=\SS^{-1}\OO_K$ and $\ppp=\SS^{-1}\qqq$ \cite[Proposition 3.16]{am1969}, 
where $\SS$ is as above. We have
\[
A/\ppp\simeq \SS^{-1}\OO_K/\SS^{-1}\qqq\simeq \SS^{-1}(\OO_K/\qqq)
\simeq \overline{\SS}^{-1}(\OO_K/\qqq)\simeq \OO_K/\qqq,
\]
where $\overline{\SS}:=\{\overline{s}\in \OO_K/\qqq: s\in \SS\}$. For the second isomorphism, see 
\cite[Corollary 3.4(iii)]{am1969} and for the third isomorphism see \cite[Chap.~3, Exercise~4]{am1969}.
The final isomorphism follows from the fact that $\OO_K/\qqq$ is a field. 

Let $R$ be either $\z$, when the
characteristic of $\OO_K$ is zero, or be $\F_{p^r}[X]$, for some finite field $\F_{p^r}$, 
when the characteristic of $A$ is $p>0$. Then $R$ is a subring of $\OO_K$, and $\OO_K$ 
is a free $R$-module of finite rank \cite[Proposition 1.36, Corollary 1.38]{keune2023}.
Let $\qqq \cap R=(\pi)$, where $\pi$ is a prime element of $R$. Then $\OO_K/\qqq$ is a finite 
extension of the finite field $k:=R/(\pi)$. Hence $\OO_K/\qqq$ is finite, which implies that $A/\ppp$ 
is finite. Now, by the induction hypothesis, assume that $A/\ppp^{\alpha-1}$ is finite. We have
\[
A/\ppp^{\alpha-1}\simeq (A/\ppp^{\alpha})/(\ppp^{\alpha-1}/\ppp^{\alpha}).
\]
On the one hand, $\ppp^{\alpha-1}/\ppp^{\alpha}$ is a principal ideal of $A/\ppp^{\alpha}$;
on the other hand, it is an $A/\ppp$-vector space. Hence $\ppp^{\alpha-1}/\ppp^{\alpha}\simeq A/\ppp$
as $A/\ppp$-vector space. This implies that $\ppp^{\alpha-1}/\ppp^{\alpha}$
is finite. Since $A/\ppp^{\alpha}$ fits into the short exact sequence
\[
0 \arr \ppp^{\alpha-1}/\ppp^{\alpha} \arr A/\ppp^{\alpha} \arr A/\ppp^{\alpha-1} \arr 0,
\]
it follows that $A/\ppp^{\alpha}$ is finite.
\end{proof}

The subsequent result generalizes Dirichlet's Unit Theorem \cite[Theorem~5.35]{keune2023}.

\begin{thm}[$S$-unit Theorem, Dirichlet, Hasse, Chevalley]\label{dir}
Let $A=\OO_{K,S}$ be a Dedekind domain of arithmetic type. Denote by $\aa$ the group of units of 
$A$ and let $\mu(A)$ be the group of roots of unity in $A$. Then
\[
\aa \simeq \mu(A) \oplus \z^{|S|-1},
\]
In particular, $\aa$ is always a finitely generated abelian group. Moreover, $A$ has infinitely 
many units if and only if $|S|\geq 2$.
\end{thm}
\begin{proof}
See \cite[Theorem 5.3.10]{weiss1976} for the general case and \cite[Theorem 6.31]{keune2023} 
for the number field case.
\end{proof}

In characteristic zero, it is straightforward to verify that the unit group $\aa=\OO_{K,S}^\times$ 
is finite if and only if $K$ is either the field of rational numbers or a totally imaginary quadratic 
field, and $S$ consists of a single infinite prime.

%\textcolor{blue}{In positive characteristic $p$, if $|S| = 1$, it follows that $A$ is the coordinate 
%ring of an affine curve obtained by removing a point from a smooth projective curve over a finite 
%field $k$ of characteristic $p>0$ \cite[\S 3]{serre1970}. Since $\aa$ comprises rational functions 
%on the curve, where the denominator is not zero at any point on the curve, it means that 
%$\aa=\OO_{K,S}^\times$ is finite if and only if $K$ is the field of fractions of some polynomial 
%ring, for example, $K = k(T)$ and $A = k[T]$.}

Now, let $K$ be a number field and let $p$ be a prime number. The prime decomposition of the ideal 
$p\OO_K$ in the ring of integers $\OO_K$ is given by
\[
p\OO_K=\ppp_1^{e_{\ppp_1}}\cdots \ppp_k^{e_{\ppp_k}},
\]
where each exponent $e_{\ppp_i}\geq 1$. It is a well-known result that the degree of the field extension 
$[K:\q]$ satisfies the fundamental identity:
\[
[K:\q]=\sum_{i=1}^k e_{\ppp_i}[\OO_K/\ppp_i: \F_p]
\]
(see \cite[Theorem 3.4]{keune2023}). A prime ideal $\ppp$ in $\OO_K$ divides $p$ if and only if 
$\ppp$ is one of the $\ppp_i$. The exponent $e_{\ppp_i}$ is called the {\it ramification index} of $\ppp_i$
over $p$, and the degree $[\OO_K/\ppp_i: \F_p]$ is called the {\it inertia degree} of $\ppp_i$ over $p$.

%%%%%%%%%%%%%%%%%%%%%%%%%%%%%%%%%%%%%%%%%%%%%%%%%%%%%%%%%%%%%%%%%%%%%%%%%%%%%%%%%%
\section{\texorpdfstring{$\GE_2$}{Lg}-rings}
%%%%%%%%%%%%%%%%%%%%%%%%%%%%%%%%%%%%%%%%%%%%%%%%%%%%%%%%%%%%%%%%%%%%%%%%%%%%%%%%%%

Let $A$ be a commutative ring with $1\neq 0$. The special linear group $\SL_2(A)$ 
consists of all $2 \times 2$ matrices over $A$ with determinant equal to one. It 
is straightforward to check that 
the center of $\SL_2(A)$ is $\mu_2(A)I_2$, where $I_2$ is the identity matrix and 
\[
\mu_2(A):=\{a\in A:a^2=1\}.
\]
Note that if $A$ is a domain, then $\mu_2(A)=\{\pm1\}$. 
Furthermore, $\SL_2$ is a functor from the category of commutative rings to the category 
of groups. 

\begin{lem}\label{A-B}
If $A$ and $B$ are two commutative rings, then
\[
\SL_2(A\times B)\simeq \SL_2(A) \times \SL_2(B).
\]
\end{lem}
\begin{proof}
It is straightforward to verify that the map 
\[
\SL_2(A) \times \SL_2(B) \arr \SL_2(A\times B),
\]
given by 
\[
\Bigg(\begin{pmatrix}
a_{11} & a_{12}\\
a_{21} & a_{22}
\end{pmatrix},
\begin{pmatrix}
b_{11} & b_{12}\\
b_{21} & b_{22}
\end{pmatrix}
\Bigg)
\mapsto
\begin{pmatrix}
(a_{11}, b_{11}) & (a_{12}, b_{12})\\
(a_{21}, b_{21}) & (a_{22}, b_{22})
\end{pmatrix},
\]
is an isomorphism of groups.
\end{proof}

Define the elementary subgroup $\Ee_2(A) \subseteq \SL_2(A)$ to be the 
subgroup generated by the elementary matrices:
\[
E_{12}(a):=\begin{pmatrix}
1 & a\\
0 & 1
\end{pmatrix} \ \ \ \text{ and } \ \ \ 
E_{21}(a):=\begin{pmatrix}
1 & 0\\
a & 1
\end{pmatrix},
\]
for all $a \in A$. The ring $A$ is called a $\GE_2$-{\it ring} if $\Ee_2(A)=\SL_2(A)$.

\begin{prp}[Cohn]\label{cohn}
{\rm (i)} Semilocal rings are $\GE_2$-rings. 
\par {\rm (ii)} Euclidean domains are $\GE_2$-rings. 
\end{prp}
\begin{proof}
The first claim is established in \cite[p. 245]{silv1982}, while the second can be found in 
\cite[\S2]{cohn1966}.
\end{proof}

\begin{cor}\label{f-ring}
Any finite ring is a $\GE_2$-ring.
\end{cor}
\begin{proof}
Any finite ring has a finite number of maximal ideals and so is semilocal.
%Let $A$ be a finite ring. Then $A$ is Artinian and so it is isomorphic to a finite product of finite local 
%rings \cite[Theorem 8.7]{am1969}. This shows that $A$ has a is semilocal. 
Now, the claim follows from 
Proposition~\ref{cohn}. 
\end{proof}

To prove our main results, we require the following lemma.

\begin{prp}[\cite{{B-E2024}}]\label{local}
Let $A$ be a commutative local ring with maximal ideal $\mmm_A$. Then
\[
\SL_2(A)^\ab\simeq
\begin{cases}
A/\mmm_A^2 &  \text{if $|A/\mmm_A|=2$}  \\
A/\mmm_A &  \text{if $|A/\mmm_A|=3$.}  \\
0 &  \text{if $|A/\mmm_A|\geq  4$}  \\
\end{cases}
\]
\end{prp}
\begin{proof}
See \cite[Proposition 4.1]{B-E2024}.
\end{proof}

The following result plays a key role in our analysis.

\begin{thm}[Vaserstein, Liehl]\label{VL}
Any Dedekind domain of arithmetic type with infinitely many units is a $\GE_2$-ring.
\end{thm}
\begin{proof}
See \cite[Theorem, page 321]{vas1972} and \cite[\S4]{liehl1981}. 
\end{proof}

\begin{lem}\label{surj-GE2}
Let  $f: A \arr B$ be a surjective homomorphism of commutative rings. If $B$ is a 
$\GE_2$-ring, then the natural map $f_\ast:\SL_2(A) \arr \SL_2(B)$ is surjective.
\end{lem}
\begin{proof}
Since $B$ is a $\GE_2$-ring, $\SL_2(B)$ is generated by the elementary matrices $E_{12}(b)$
and $E_{21}(b)$, $b\in B$. If $b=f(a)$, then $f_\ast(E_{ij}(a))=E_{ij}(b)$.
\end{proof}

Let $I$ be an ideal of $A$ and let $\pi:A \arr A/I$, $a \mapsto \bar{a}$, be the natural 
quotient map. Let $\Gamma(A,I)$ be the kernel of the natural map
\[
\pi_\ast:\SL_2(A) \arr \SL_2(A/I). 
\]
Hence
\[
\Gamma(A, I) :=\Bigg\{ {\mtxx a b c d} \in \SL_2(A) :b, c, a-1, d-1 \in I\Bigg\}.
\]

\begin{lem}\label{extension}
Let $I$ be an ideal of a ring $A$. If $A/I$ is a $\GE_2$-ring, then we have the extension
\[
1 \arr \Gamma(A,I) \arr \SL_2(A) \arr \SL_2(A/I) \arr 1.
\]
\end{lem}
\begin{proof}
Since $A/I$ is a $\GE_2$-ring, by Lemma \ref{surj-GE2}, the map $\pi_\ast:$ $\SL_2(A) \arr \SL_2(A/I)$
is surjective. By definition, the kernel of this map is $\Gamma(A,I)$, which gives us the desired extension.
\end{proof}

Let $A$ be a Dedekind domain of arithmetic type. A subgroup of the form $\Gamma(A,I)$, 
for some nontrivial ideal $I$, is called a {\it principal congruence subgroup} of $\SL_2(A)$. 
A subgroup of $\SL_2(A)$ is said to be a {\it congruence subgroup} if it contains a principal 
congruence subgroup. The following theorem is due to Serre \cite{serre1970}.

\begin{thm}[Congruence Subgroup Property for $\SL_2$, Serre]\label{csp}
Let $A$ be a Dedekind domain of arithmetic type with infinitely many units. Then any 
non-central normal subgroup $N$ of $\SL_2(A)$ is a congruence subgroup. In particular, 
the index of $N$ in $\SL_2(A)$ is finite.
\end{thm}
\begin{proof}
See \cite[Proposition 2]{serre1970}.
\end{proof}

%%%%%%%%%%%%%%%%%%%%%%%%%%%%%%%%%%%%%%%%%%%%%%%%%%%%%%%%%%%%%%%%%%%%%%%%%%%%%%%%%%%%%%%%%%%%%%%%%%%%%%%%
\section{Abelianization of \texorpdfstring{$\SL_2$}{Lg} over Dedekind domains of arithmetic type}
%%%%%%%%%%%%%%%%%%%%%%%%%%%%%%%%%%%%%%%%%%%%%%%%%%%%%%%%%%%%%%%%%%%%%%%%%%%%%%%%%%%%%%%%%%%%%%%%%%%%%%%%

Let $K$ be an algebraic number field with ring of integers $\OO_K$. Let $S_2$, $S_2'$ 
and $S_3$ be the subsets of $\Spec(\OO_K)$ as defined in the introduction. The following 
theorem is the first main result of this paper.

\begin{thm}\label{Main}
If $A=\OO_{K,S}$ is a Dedekind domain of arithmetic type of characteristic zero with 
infinitely many units, then
\[
\SL_2(A)^\ab \simeq \bigoplus_{\qqq\in S_2\backslash S} \z/4 \oplus 
\bigoplus_{\qqq\in S_2'\backslash S} (\z/2\oplus \z/2) \oplus
\bigoplus_{\qqq\in S_3\backslash S} \z/3.
\]
In particular, $\SL_2(A)^\ab$ is of exponent dividing $12$.
\end{thm}
\begin{proof}
The commutator subgroup of a group is always normal \cite[Chap.~I, \S12, Exercise~3]{lang2002}. 
Thus $[\SL_2(A), \SL_2(A)]$ is normal in $\SL_2(A)$. Moreover, $[\SL_2(A), \SL_2(A)]$ is non-central. 
In fact, the center of $\SL_2(A)$ is $\mu_2(A)I_2$. If the commutator subgroup is central,
then we have a natural surjective map 
\[
\SL_2(A)^\ab \two \PSL_2(A):=\SL_2(A)/\mu_2(A)I_2
\]
and, thus, $\PSL_2(A)$ must be abelian. But this is not true (over any ring where $1\neq 0$), 
because, in $\PSL_2(A)$, we have 
$\overline{E_{12}(1)}\ \overline{E_{21}(1)}\neq \overline{E_{21}(1)}\ \overline{E_{12}(1)}$. 
Thus, $[\SL_2(A), \SL_2(A)]$ is a non-central normal subgroup of 
$\SL_2(A)$. 

By Theorem \ref{csp}, $[\SL_2(A), \SL_2(A)]$ is a congruence subgroup. So it contains a principal 
congruence subgroup $\Gamma(A, I)$, for some non-trivial ideal $I$ of $A$. This yields 
the surjection
\begin{equation}\label{surj-0}
\SL_2(A)/\Gamma(A, I) \two \SL_2(A)/[\SL_2(A), \SL_2(A)] = \SL_2(A)^{\ab}.
\end{equation}

Since $A$ is a Dedekind domain of arithmetic type, by Lemma \ref{PFR}, $A/I$ is a finite ring.
By Corollary \ref{f-ring}, $A/I$ is a $\GE_2$-ring and so, by Lemma \ref{extension},
we have the isomorphism
\[
\SL_2(A)/\Gamma(A, I) \simeq \SL_2(A/I).
\]
From this and (\ref{surj-0}), we obtain the surjective map 
\[
\phi: \SL_2(A/I)\simeq \SL_2(A)/\Gamma(A, I) \two \SL_2(A)^\ab,
\]
which is given by 
\[
{\mtxx {\overline{a}} {\overline{b}} {\overline{c}} {\overline{d}}} \mapsto \overline{\mtxx a b c d}.
\]
Since $\SL_2(A)^\ab$ is abelian, under the above map, the commutator subgroup of $\SL_2(A/I)$, 
i.e. $[\SL_2(A/I),\SL_2(A/I)]$, maps to $1$ (see \cite[Chap.~I, \S12, Exercise~3]{lang2002}). 
Thus, we have the surjective homomorphism
\begin{equation} \label{surj}
\overline{\phi}: \SL_2(A/I)^{\ab} \two \SL_2(A)^{\ab}.
\end{equation}
Let $I':=6I$ and let $I'=\ppp_1^{\alpha_{\ppp_1}} \cdots \ppp_r^{\alpha_{\ppp_r}}$ be the prime 
decomposition of $I'$ in $A$. Thus, $\alpha_{\ppp_i}\geq 1$, for all $i$. Let 
\[
J:=(\ppp_1 \cdots \ppp_r)I'=\ppp_1^{\alpha_{\ppp_1+1}} 
\cdots \ppp_r^{\alpha_{\ppp_r+1}}\subseteq I'.
\]
Note that $\alpha_{\ppp_i}+1>1$, for all $i$, and so we may write
\[
J=\ppp_1^{\beta_{\ppp_1}} \cdots \ppp_r^{\beta_{\ppp_r}},
\]
where each $\beta_{\ppp_i}>1$. Since $J\se I$, we have the natural surjective map of finite rings 
$A/J \two A/I$ (see Lemma \ref{PFR}). By Corollary \ref{f-ring} and Lemma \ref{surj-GE2}, we have 
the natural surjective map 
\[
\psi: \SL_2(A/J) \two \SL_2(A/I).
\]
Clearly, the map
\begin{equation}\label{surj-1}
\overline{\psi}:\SL_2(A/J)^\ab \two  \SL_2(A/I)^\ab, \ \ \ 
\overline{X} \mapsto \overline{\psi(X)},
\end{equation}
is an epimorphism. From  (\ref{surj}) and (\ref{surj-1}), we obtain the epimorphism
\[
\overline{\phi}\circ \overline{\psi}:\SL_2(A/J)^{\ab} \two \SL_2(A)^{\ab}.
\]

On the other hand, by Corollary \ref{f-ring} and Lemma \ref{surj-GE2}, the natural map 
$\SL_2(A) \arr \SL_2(A/J)$ is surjective. It is straightforward to verify that the 
composition 
\[
\SL_2(A/J)^{\ab} \two \SL_2(A)^{\ab} \two \SL_2(A/J)^{\ab}
\]
is the identity map. Therefore, we conclude that
\[
\SL_2(A)^{\ab} \simeq \SL_2(A/J)^{\ab}.
\]

By the Chinese Remainder Theorem, 
$A/J \simeq A/\ppp_1^{\beta_{\ppp_1}} \times \cdots \times A/\ppp_r^{\beta_{\ppp_r}}$,
where each $A/\ppp_i^{\beta_{\ppp_i}}$ is a finite local ring. By Lemma \ref{A-B},
\[
\SL_2(A/J) \simeq \SL_2(A/\ppp_1^{\beta_{\ppp_1}}) \times \cdots \times 
\SL_2(A/\ppp_r^{\beta_{\ppp_r}}).
\]
Hence
\[
\SL_2(A)^{\ab} \simeq \SL_2(A/\ppp_1^{\beta_{\ppp_1}})^{\ab} 
\times \cdots \times \SL_2(A/\ppp_r^{\beta_{\ppp_r}})^{\ab}.
\]

Recall that $A=\OO_{K,S}$, the ring of $S$-integers, where $K$ is the field of fractions of 
$A$ and $S$ is a finite non-empty set of primes of $K$ containing all infinite primes. 

Let $\OO_K$ denote the ring of integers in $K$. There is a natural bijection between the
sets $\Spec(A)$ and $\Spec(\OO_K)\backslash S$. In fact, $ A=\OO_{K,S}=\SS^{-1}\OO_K$, where
$\mathcal{S}=\OO_K\backslash \bigcup_{\qqq\in\Spec(\OO_K)\backslash S}\qqq$ (see Section \ref{sec1}).

For a non-trivial prime ideal $\ppp\in \Spec(A)$, we denote its corresponding prime ideal 
in $\Spec(\OO_K)$ by $\qqq$. So $\qqq=\ppp\cap \OO_K$ and $\ppp=\SS^{-1}\qqq$ 
\cite[Proposition 3.16]{am1969}. From these and the properties of the localization functor 
(\cite[Proposition~3.3 and Proposition~3.11]{am1969}), we have 
\begin{align*}
A/\ppp_i^{\beta_{\ppp_i}} &\simeq (\mathcal{S}^{-1}\OO_K)/(\mathcal{S}^{-1}\qqq_i)^{\beta_{\ppp_i}}\\
&\simeq (\mathcal{S}^{-1}\OO_K)/(\mathcal{S}^{-1}(\qqq_i^{\beta_{\ppp_i}}))\\
& \simeq \mathcal{S}^{-1}(\OO_K/\qqq_i^{\beta_{\ppp_i}}).
\end{align*} 
%Observe that 
%\[
%A/\ppp_i\simeq \OO_K/\qqq_i,
%\]
%(see the 

Let $|A/\ppp_i|\geq 4$, for some $i$. The ring $A/\ppp_i^{\beta_{\ppp_i}}$ is local with the 
maximal ideal $\ppp_i/\ppp_i^{\beta_{\ppp_i}}$ and the quotient field $A/\ppp_i$. Then, 
by Proposition \ref{local}, we have 
\[
\SL_2(A/\ppp_i^{\beta_{\ppp_i}})^\ab=0.
\]
Therefore, we may assume that, for all $i$, either $A/\ppp_i\simeq\F_2$ or $A/\ppp_i\simeq\F_3$. 
Observe that if $\qqq\in S$, then clearly $\SS^{-1}\qqq=\SS^{-1}\OO_K$ 
\cite[Proposition 3.11(ii)]{am1969} and thus, for any $r\geq 1$, 
$(\mathcal{S}^{-1}\OO_K)/(\mathcal{S}^{-1}\qqq)^r=0$.

Combining all of the above results, we obtain the decomposition:
\begin{align*}
\SL_2(A)^\ab \simeq &\bigoplus_{\qqq\in S_2\backslash S} 
\SL_2(\mathcal{S}^{-1}(\OO_K/\qqq^{\beta_\ppp}))^{\ab} \oplus 
\bigoplus_{\qqq\in S_2'\backslash S} 
\SL_2(\mathcal{S}^{-1}(\OO_K/{\qqq}^{\beta_{\ppp}}))^{\ab} \\
&\oplus\bigoplus_{\qqq\in S_3\backslash S} 
\SL_2(\mathcal{S}^{-1}(\OO_K/\qqq^{\beta_\ppp}))^{\ab}.
\end{align*}
%where $\qqq=\ppp \cap \OO_K$, with $\ppp$ belonging to the set $\{\ppp_1, \dots, \ppp_r\}$.
Note that $\beta_\ppp>1$, for all $\ppp$, and have in mind that $\OO_K/\qqq\simeq A/\ppp$ 
(see the proof of Lemma \ref{PFR}). Now, we apply Proposition \ref{local}:\\
If $\OO_K/\qqq\simeq \F_3$, then
\[
\SL_2(\SS^{-1}(\OO_K/\qqq^{\beta_\ppp}))^{\ab}
\simeq \SL_2(A/\ppp^{\beta_\ppp})^{\ab}
\simeq (A/\ppp^{\beta_\ppp})/(\ppp/\ppp^{\beta_\ppp})
\simeq A/\ppp\simeq \z/3.
\]
If $\OO_K/\qqq\simeq \F_2$, then
\[
\SL_2(\SS^{-1}(\OO_K/\qqq^{\beta_\ppp}))^{\ab}
\simeq \SL_2(A/\ppp^{\beta_\ppp})^{\ab}
\simeq (A/\ppp^{\beta_\ppp})/(\ppp/\ppp^{\beta_\ppp})^2
\simeq A/\ppp^2.
\]
To complete the proof, it remains to study the additive group $A/\ppp^2$, for $\ppp=\SS^{-1}\qqq$, where
$\qqq\in S_2 \cup S_2'$.
Recall that for these primes, we have $A/\ppp\simeq \F_2$. Now, consider the exact sequence
\[
0 \arr \ppp/\ppp^2 \arr A/\ppp^2 \arr \F_2 \arr 0.
\]
Since $\ppp/\ppp^2$ is a principal ideal of $A/\ppp^2$, it is a one-dimensional $A/\ppp$-vector space. Hence 
$\ppp/\ppp^2\simeq \F_2$. It follows that $A/\ppp^2$ has order $4$.

Let $\qqq\in S_2'$. If $e_\qqq$ denotes the ramification index of $\qqq$ over the prime $2$, then
$e_\qqq>1$.  This implies that $2 \in \qqq^{e_\qqq}$ and so $2 \in \qqq^2$. This shows that $2 \in \ppp^2$.
Now, for any $x+\ppp^2\in A/\ppp^2$, we have
\[
2(x+\ppp^2)=(x+\ppp^2)+(x+\ppp^2)=2x+\ppp^2=\ppp^2.
\]
Thus, every element of the additive group $A/\ppp^2$ is of order at most two, so 
\[
A/\ppp^2\simeq\z/2\oplus\z/2.
\]
Now assume that $\qqq\in S_2$. Then $e_\qqq=1$. This implies that $2\notin \qqq^2$ 
%(\textcolor{blue}{I guess that it is $2 \notin \qqq^2$, since $2 \in \qqq$ by the 
%condition $\qqq$ over $2$, but the ramification index is $1$, so, follows 
%$2 \notin \qqq^2$, hence $2 \notin \ppp^2$ as in the sequence})
and so $2 \notin \ppp^2$. Hence
\[
2(1+\ppp^2)=2+\ppp^2 \neq \ppp^2.
\]
This shows that $1+\ppp^2$ has order $4$, and therefore  
\[
A/\ppp^2 \simeq \z/4.
\]
This completes the proof of the theorem.
\end{proof}

Let $K$ be a global field of positive characteristic, and let $\OO_K$ denote its ring of 
integers. Consider the subsets $S_2''$ and $S_3''$ of $\Spec(\OO_K)$, as defined in the 
introduction. We now state our second main result.

\begin{thm}\label{main2}
Let $A=\OO_{K,S}$ be a Dedekind domain of arithmetic type of positive characteristic $p$
such that $[K:\F_q(t)]<\infty$, where $q=p^r$ for some natural number $r$. If $A$ has infinitely many 
units, then
\[
\SL_2(A)^\ab \simeq \begin{cases}
\bigoplus_{\qqq\in S_2''\backslash S}(\z/2\oplus \z/2) &  \text{if $q=2$}\\
\bigoplus_{\qqq\in S_3''\backslash S} \z/3  & \text{if $q=3$.}\\
0 & \text{if $q\geq 4$}
\end{cases}
\]
In particular, $\SL_2(A)^\ab$ is of exponent dividing $6$.
\end{thm}
\begin{proof}
The proof proceeds similarly to that of Theorem~\ref{Main}. The argument up to (\ref{surj})
is the same. Then we take 
\[
I':=\begin{cases}
t(t-1)I & \text{if $q=2$}\\
t(t-1)(t-2)I & \text{if $q=3$}\\
I & \text{if $q\geq 4$}
\end{cases}
\]
and define $J$ in a similar way. If $J=\ppp_1^{\beta_{\ppp_1}} \cdots \ppp_r^{\beta_{\ppp_r}}$
is the prime decomposition of $J$, then 
\[
\SL_2(A)^{\ab} \simeq \SL_2(A/J)^{\ab} \simeq \SL_2(A/\ppp_1^{\beta_{\ppp_1}})^{\ab} 
\times \cdots \times \SL_2(A/\ppp_r^{\beta_{\ppp_r}})^{\ab}.
\]
Note that $\beta_{\ppp_i}>1$ for all $i$. The remainder of the proof of Theorem~\ref{main2} follows in the 
same way. Observe that for any nontrivial prime ideal $\ppp$ of $A$, the field $\F_q$ embeds 
into $A/\ppp$, and hence $|A/\ppp|\geq q$. Moreover, if $|A/\ppp|\geq 4$, then for every positive 
integer $\beta$, $\SL_2(A/\ppp^\beta)^\ab=0$. Let 
\[
\qqq_i:=\OO_K\cap \ppp_i.
\]

If $q\geq 4$, then $\SL_2(A/\ppp_i^{\beta_{\ppp_i}})^\ab=0$, for any $1\leq i\leq r$. Hence 
\[
\SL_2(A)^\ab=0.
\]

If $q=3$, then, as in the proof of Theorem \ref{Main}, we have
\begin{align*}
\SL_2(A)^\ab \simeq \bigoplus_{\qqq\in S_3''\backslash S} 
\SL_2(\mathcal{S}^{-1}(\OO_K/\qqq^{\beta_\ppp}))^{\ab} 
\simeq \bigoplus_{\qqq\in S_3''\backslash S} \z/3.
\end{align*}

If $q=2$, then 
\begin{align*}
\SL_2(A)^\ab \simeq \bigoplus_{\qqq\in S_2''\backslash S} 
\SL_2(\mathcal{S}^{-1}(\OO_K/\qqq^{\beta_\ppp}))^{\ab} 
\simeq \bigoplus_{\qqq\in S_2''\backslash S} A/\ppp^2.
\end{align*}
In this case, $A/\ppp^2$ is a two-dimensional $\F_2$-vector space of dimension $2$. Consequently, 
as abelian groups, we have an isomorphism
\[
A/\ppp^2 \simeq \z/2\oplus \z/2.
\]
This completes the proof of the theorem.
\end{proof}

\begin{exa}
The ring $A := \z[\sqrt[3]{5}]$ is the ring of integers of $K=\q(\sqrt[3]{5})$ 
(see \cite[Example~3.9]{keune2023}). There are two infinite primes in $K$; one associated
to the real embedding of $K$ and another one associated to two complex embeddings of $K$.
Hence $A=\OO_K=\OO_{K, S}$, where $S$ is the set of two infinite primes of $K$.
It follows from Theorem \ref{dir} that $A^{\times}$ is of rank $1$, so it is infinite.
Hence, by Theorem \ref{Main}, we can determine the structure of $\SL_2(A)^{\ab}$. Let 
$\alpha := \sqrt[3]5$. We have the prime decompositions: $2A=\ppp_1\ppp_2$ and $3A=\ppp^3$, 
where
\[
\ppp_1=(2,1+\alpha), \ \ \ \ppp_2=(2,1+\alpha+\alpha^2), \ \ \ \ppp=(3,1+\alpha)
\]
(see \cite[Example 3.9]{keune2023}). 
Thus, $S_3=\{\ppp\}$. Let $f_1$ and $f_2$ denote the inertia degrees of $\ppp_1$ and 
$\ppp_2$, respectively. It is easy to verify that $f_1=1$ and $f_2=2$. Hence we have 
$S_2^{\prime}=\varnothing$, and  $S_2=\{\ppp_1\}$. Applying Theorem~\ref{Main}, 
we conclude that
\[
\SL_2(\z[\sqrt[3]{5}])^{\ab} \simeq \z/4 \oplus \z/3\simeq \z/12.
\]
\end{exa}

\begin{exa}\label{Z1-n}
Let $n>1$ be an integer. It is known that $A:=\z[1/n]$ is a Euclidean domain. Let $p$ be a prime number. 
If $p\mid n$, then $pA=A$, while if $p\nmid n$, the ideal $pA$ is a non-zero prime ideal of $A$. 
Thus,
\[
S_2=\begin{cases}
\{2A\} & \text{if $2\nmid n$}\\
\varnothing & \text{if $2\mid n$}
\end{cases}, \ \ \ \ S_2'=\varnothing, \ \ \ \ 
S_3=\begin{cases}
\{3A\} & \text{if $3\nmid n$}\\
\varnothing & \text{if $3\mid n$}
\end{cases}.
\]
Observe that, $A=\OO_{\q,S}$, where $S=\{|\ |, (p_1), \dots, (p_k)\}$ and $n=p_1^{r_1}\cdots p_k^{r_k}$ 
is the prime decomposition of $n$. Applying Theorem \ref{Main}, we obtain
\[
\SL_2(\z[1/n])^\ab \simeq 
\begin{cases}
0   & \text{if $2\mid n$, $3\mid n$}  \\
\z/3   & \text{if $2\mid n$, $3\nmid n$}\\
\z/4   & \text{if $2\nmid n$, $3\mid n$}\\
\z/12 &  \text{if $2\nmid  n$, $3\nmid n$}
\end{cases}.
\]
This result was previously proved independently, using different and more elementary methods, in 
\cite[Theorem~1.2]{carl2024} and \cite[Proposition~4.4]{B-E2024}.
\end{exa}

The present paper grew out of our desire to generalize the isomorphism of Example~\ref{Z1-n} to the 
ring of integers of quadratic fields. This is the topic of the next section.

%%%%%%%%%%%%%%%%%%%%%%%%%%%%%%%%%%%%%%%%%%%%%%%%%%%%%%%%%%%%%%%%%%%%%%%%%%%%%%%%%%%%%%%%%%%%%%%%%%%%
\section{Abelianization of \texorpdfstring{$\SL_2$}{Lg} over rings of integers of quadratic fields}
%%%%%%%%%%%%%%%%%%%%%%%%%%%%%%%%%%%%%%%%%%%%%%%%%%%%%%%%%%%%%%%%%%%%%%%%%%%%%%%%%%%%%%%%%%%%%%%%%%%%

Let $K$ be a quadratic field, i.e., $[K:\q]=2$. Then $K=\q(\sqrt{d})$, for some square-free integer $d$.
Let $\OO_d$ be the ring of integers of $\q(\sqrt{d})$. Let $p\in\N$ be a prime and consider the 
ideal $p\OO_d$ in $\OO_d$. Then
\begin{itemize}
\item[(i)]  $p$ is called {\it inert} in $\OO_d$ if $p\OO_d$ remains a prime ideal in $\OO_d$;
\item[(ii)] $p$ is called {\it split} in $\OO_d$ if $p\OO_d$ is the product of two distinct prime ideals;
\item[(iii)] $p$ is called {\it ramified} in $\OO_d$ if $p\OO_d$ equals the square of a prime ideal.
\end{itemize}

\begin{prp}[{Corollary to \cite[Theorem 3.7]{{keune2023}}}] \label{2-d}
Let $K=\q(\sqrt{d})$ be a quadratic field, where $d$ is a square free integer. Then
\par {\rm (i)} $2$ is inert if $d\equiv 5 \pmod 8$, split if $d\equiv 1 \pmod 8$ 
and ramified otherwise.
\par {\rm (ii)} $3$ is inert if $d\equiv 2 \pmod 3$, split if $d\equiv 1 \pmod 3$
and ramified if $d\equiv 0 \pmod 3$.
\end{prp}
\begin{proof}
For (i), see \cite[Theorem 3.7]{keune2023}.
If $d \equiv 0 \pmod 3$ , then $3 \mid d$, and thus $3$ ramifies by \cite[Theorem~3.7]{keune2023}. 
Now, suppose that $3 \nmid d$. Then $d \equiv n^2 \pmod 3$ if and only if $d \equiv 1 \pmod 3$. 
Hence, by \cite[Theorem 3.7]{keune2023}, the prime $3$ splits if and only if $d \equiv 1 \pmod 3$. 
Finally, $3$ is inert if and only if $d \equiv 2 \pmod 3$.   
\end{proof}

\begin{thm}\label{d>0}
Let $K=\q(\sqrt{d})$ be a quadratic field, where $d$ is a square-free positive integer. 
If $\OO_d$ is the ring of integers of $K$, then
\[
\SL_2(\mathcal{O}_d)^\ab\simeq \begin{cases}
\ 0 & \text{if $d\equiv 5 \pmod  {24}$}\\
(\z/4)^2 \oplus (\z/3)^2 &\text{if $d\equiv 1 \pmod  {24}$}\\
(\z/4)^2 \oplus \z/3 &\text{if $d\equiv 9 \pmod  {24}$}\\
(\z/3)^2 & \text{if $d\equiv 13 \!\!\!\pmod  {24}$}\\
\ \z/3 &\text{if $d\equiv 21 \!\!\!\pmod  {24}$}\\
(\z/4)^2 &\text{if $d\equiv 17 \!\!\!\pmod  {24}$}\\
(\z/2)^2  &\text{if $d\equiv 2, 11, 14, 20, 23 \!\!\!\pmod  {24}$}\\
(\z/2)^2 \oplus (\z/3)^2 &\text{if $d\equiv 4,7,10,19, 22 \pmod  {24}$}\\
(\z/2)^2 \oplus \z/3 &\text{otherwise.}
\end{cases}
\]
\end{thm}
\begin{proof}
First, note that $\OO_d=\OO_{K,S}$, where $S$ consists of two infinite primes, associated to two
real embeddings of $K$. Hence, by
Theorem \ref{dir}, the unit group $\OO_d^\times$ is infinite. By Theorem \ref{Main}, we obtain 
an isomorphism $\SL_2(\OO_d)^\ab\simeq G_1\oplus G_2$, where
\[
G_1= \begin{cases}
\ 0 & \text{if $2$ is inert}\\
\z/4\oplus \z/4 & \text{if $2$ is split}\\
\z/2\oplus \z/2 & \text{if $2$ is ramified,}
\end{cases}
\ \
G_2= \begin{cases}
\ 0 & \text{if $3$ is inert}\\
\z/3\oplus \z/3 & \text{if $3$ is split}\\
\z/3 & \text{if $3$ is ramified.}
\end{cases}
\]
Now, the claim follows directly from Proposition~\ref{2-d}.
\end{proof}

The case where $K=\q(\sqrt{-d})$ differs significantly.
In this setting, the ring of integers $\OO_{-d}$ contains only finitely many units. 
In most cases, it is even not a $\GE_2$-ring. The ring $\OO_{-d}$ is a $\GE_2$-ring
if and only if $d=1,2,3,7,11$ \cite[Theorem 6.1]{cohn1966}. The structure of the 
abelianization $\SL_2(\OO_{-d})^\ab$ has been explicitly computed by Cohn for 
these values of $d$ (see \cite[page 162]{cohn1968} and \cite[Proposition 4.7]{B-E2024}). 
For additional values of $d$, further results can be found in \cite{swan1971}, 
\cite{rahm2013}, and \cite{rahm-2013}. However, the structure of 
$\SL_2(\OO_{-d})^\ab$ remains unknown for many values of $d$.
%and \cite{rf2011}.

Now, let $S$ be a set of primes containing at least two elements, including the infinite prime. 
In this setting, Theorem~\ref{Main} can be applied to compute $\SL_2(\OO_{-d, S})^\ab$. If the
prime $2$ splits in $\mathcal{O}_{-d}$, then it factors as $2\mathcal{O}_{-d}=\ppp_1\ppp_2$; 
if it is ramified, we have $2\mathcal{O}_{-d}=\ppp^2$. Similarly, if $3$ splits, then 
$3\mathcal{O}_{-d}=\qqq_1\qqq_2$, and if it is ramified, $3\mathcal{O}_{-d}=\qqq^2$. Let
\begin{align*}
&\beta_{\ppp_i}:=
\begin{cases}
0 & \text{if } \ppp_i \in S\\
1 & \text{ if } \ppp_i \notin S
\end{cases},
\ \ \ \ \     
\beta_{\ppp}:=
\begin{cases}
0 & \text{if } \ppp \in S\\
1 & \text{ if } \ppp \notin S
\end{cases},\\
&\beta_{\qqq_i}:=
\begin{cases}
0 & \text{if } \qqq_i \in S\\
1 & \text{ if } \qqq_i \notin S
\end{cases},
\ \ \ \ \ 
\beta_{\qqq}:=
\begin{cases}
0 & \text{if } \qqq \in S\\
1 & \text{ if } \qqq \notin S
\end{cases}.
\end{align*}
These notations allow us to derive a result analogous to Theorem~\ref{d>0}.

\begin{thm}\label{d<0}
Let $K=\q(\sqrt{-d})$, where $-d$ is a square-free negative integer, and let $\mathcal{O}_{-d}$
denote the ring of integers of $K$. Let $S$ be a finite set of primes of $\mathcal{O}_{-d}$
containing at least two elements, including the infinite prime. Then 
\[
\SL_2(\mathcal{O}_{-d,S})^\ab \simeq \begin{cases}
\ \ 0 & \text{if $-d\equiv 5 \pmod  {24}$}\\
(\z/4)^{\beta_{\ppp_1}+\beta_{\ppp_2}}
%\oplus (\z/4)^{\beta_{\ppp_2}} 
\oplus (\z/3)^{\beta_{\qqq_1}+\beta_{\qqq_2}} 
%\oplus (\z/3)^{\beta_{\qqq_2}} 
&\text{if $-d\equiv 1 \pmod  {24}$}\\
(\z/4)^{\beta_{\ppp_1}+\beta_{\ppp_2}}
%{\beta_{\ppp_1}}\oplus (\z/4)^{\beta_{\ppp_2}} 
\oplus (\z/3)^{\beta_{\qqq}} &\text{if $-d\equiv 9 \pmod  {24}$}\\
(\z/3)^{\beta_{\qqq_1}+\beta_{\qqq_2}} 
%\oplus (\z/3)^{\beta_{\qqq_2}} 
& \text{if $-d\equiv 13 \!\!\!\pmod  {24}$}\\
(\z/3)^{\beta_{\qqq}} &\text{if $-d\equiv 21 \!\!\!\pmod  {24}$}\\
(\z/4)^{\beta_{\ppp_1}+\beta_{\ppp_2}} 
%\oplus (\z/4)^{\beta_{\ppp_2}} 
&\text{if $-d\equiv 17 \!\!\!\pmod  {24}$}\\
(\z/2)^{2\beta_{\ppp}}  &\text{if $-d\equiv 2, 11, 14, 20, 23 \!\!\!\!\!\pmod  {24}$}\\
(\z/2)^{2\beta_{\ppp}} \oplus (\z/3)^{\beta_{\qqq_1}+\beta_{\qqq_2}} 
%\oplus (\z/3)^{\beta_{\qqq_2}} 
&\text{if $-d\equiv 4,7,10,19, 22 \! \pmod  {24}$}\\
(\z/2)^{2\beta_{\ppp}} \oplus (\z/3)^{\beta_{\qqq}} &\text{otherwise.}
\end{cases}
\]
\end{thm}
\begin{proof}
Since $S$ contains at least two primes, Theorem~\ref{dir} implies that $\mathcal{O}_{d,S}^{\times}$
is infinite. As in Theorem~\ref{d>0}, applying Theorem~\ref{Main} yields 
an isomorphism $\SL_2(\mathcal{O}_{d,S})^{\ab} \simeq G_1 \oplus G_2$, where
\[
G_1= \begin{cases}
\ \ 0 & \text{if $2$ is inert}\\
(\z/4)^{\beta_{\ppp_1}}\oplus (\z/4)^{\beta_{\ppp_2}} & \text{if $2$ is split}\\
(\z/2\oplus \z/2)^{\beta_{\ppp}} & \text{if $2$ is ramified,}
\end{cases}
\]
\[
G_2= \begin{cases}
\ \ 0 & \text{if $3$ is inert}\\
(\z/3)^{\qqq_1}\oplus (\z/3)^{\qqq_2} & \text{if $3$ is split}\\
(\z/3)^{\qqq} & \text{if $3$ is ramified.}
\end{cases}
\]
%Here, $\ppp_1$ and $\ppp_2$ denote the prime ideals appearing in the factorization of 
%$2$ in the case where $2$ splits, while $\ppp$ denotes the (ramified) prime ideal when 
%$2$ is ramified. Similarly, $\qqq_1$ and $\qqq_2$ correspond to the primes in the 
%factorization of $3$ when it splits, and $\qqq$ denotes the ramified prime when $3$
%is ramified in $\mathcal{O}_{-d,S}$. 
The claim now follows from Proposition~\ref{2-d}.
\end{proof}

\begin{exa}
Let $d=-15$. Then $\OO_{-15}=\z[(1+\sqrt{-15})/2]$. By a result of Swan, 
\[
\SL_2(\OO_{-15})^\ab\simeq \z/12\oplus \z\oplus \z
\]
(see \cite[Corollary 15.2]{swan1971}). Since $-15\equiv 1 \pmod 8$, the prime $2$ splits in $\OO_{-15}$,
while $-15\equiv 0 \pmod 3$, so $3$ is ramified. In fact, we have the decompositions $2\OO_{-15}=\ppp_1\ppp_2$
and $3\OO_{-15}=\qqq^2$, where
\[
\ppp_1=(2, (1+\sqrt{-15}),\ \ \  \ppp_2=(2, (1-\sqrt{-15}), \ \ \ \qqq=(3, \sqrt{-15}).
\]
Therefore, if $S$ is any finite set of primes of $\OO_{-15}$ containing at least two elements, including the 
infinite prime, then, by the above theorem, we obtain
\[
\SL_2(\OO_{-15, S})^\ab\simeq (\z/4)^{\beta_{\ppp_1}+\beta_{\ppp_2}}
\oplus (\z/3)^{\beta_{\qqq}}.
\]
\end{exa}

\begin{rem}
We can show that the natural homomorphism 
\[
\SL_2(\OO_{-d})^\ab \arr \SL_2(\OO_{-d, S})^\ab
\]
is always surjective when $S$ contains at least the infinite prime. However, we do not provide a proof 
here, as this fact will not be used in the sequel.
\end{rem}

%%%%%%%%%%%%%%%%%%%%%%%%%%%%%%%%%%%%%%%%%%%%%%%%%%%%%%%%%%%%%%%%%%%%%%%%%%%%%%%%%%%%%%%%%%%%%
\section{Abelianization of \texorpdfstring{$\SL_2$}{Lg} over rings of integers of Galois extensions}
%%%%%%%%%%%%%%%%%%%%%%%%%%%%%%%%%%%%%%%%%%%%%%%%%%%%%%%%%%%%%%%%%%%%%%%%%%%%%%%%%%%%%%%%%%%%%

Let $K$ be a Galois extension of $\q$ of degree $n=[K:\q]$ and let $\mathcal{O}_K$ be its ring of 
integers. If $n=1$, then $K=\q$ and so $\OO_K=\z$. It is known that 
\[
\SL_2(\z)^\ab\simeq \z/12
\]
(see \cite[Example 3.4]{B-E2024}). If $n=2$, then $K$ is a quadratic field and $\SL_2(\OO_K)$ is 
discussed in the previous section.

So, let $n\geq 3$. By \cite[Theorem 3.11, Corollary 3.12]{keune2023}, for any prime $p \in \z$, $p\OO_K$ 
has the prime decomposition in $\mathcal{O}_K$,
\[
p\mathcal{O}_K \simeq (\mathfrak{p}_1\mathfrak{p}_2\cdots \mathfrak{p}_{m_p})^{e_p},
\]
where all primes $\ppp_i$ have the same ramification index $e_p$ and the same inertia degree
$f_p$, i.e., $[\OO_K/\ppp_i:\z/p]=f_p$, for all $1\leq i\leq m_p$. Observe that $n=e_pf_pm_p$.
 
\begin{thm}\label{Galois}
Let $K$ be a Galois extension of $\q$ of degree $n\geq 3$ and let $\mathcal{O}_K$ be the ring of 
integers of $K$. If
\[
2\mathcal{O}_K=(\ppp_1\ppp_2\cdots \ppp_m)^{e_2},\ \ \ \ \ \ 3\mathcal{O}_K=(\qqq_1\qqq_2\cdots \qqq_k)^{e_3}
\]
are the prime decompositions of $2\OO_K$ and $3\OO_K$, with $e_2$ and $e_3$ being the ramification indexes 
and $f_2$ and $f_3$ the corresponding inertia degrees, then
\begin{align*}
\SL_2(\mathcal{O}_K)^{\ab}\simeq
\begin{cases}
0 &\text{if $f_2>1,f_3>1$}\\
(\z/3)^{\frac{n}{e_3}} &\text{if $f_2>1, f_3=1$}\\
(\z/2\oplus \z/2)^{\frac{n}{e_2}} &\text{if $f_2=1,e_2>1, f_3>1$}\\
(\z/2\oplus \z/2)^{\frac{n}{e_2}}\oplus (\z/3)^{\frac{n}{e_3}} &\text{if $f_2=1, e_2>1,f_3=1$.}\\
(\z/4)^{n} &\text{if $f_2=1,e_2=1, f_3>1$}\\
(\z/4)^{n} \oplus (\z/3)^{\frac{n}{e_3}} &\text{if $f_2=1, e_2=1, f_3=1$}
\end{cases}
\end{align*}
\end{thm}
\begin{proof}
Consider the sets $S_2$, $S_2'$ and $S_3$ as defined in the introduction.
\par {\bf Case (i)} $f_2>1$: Here, we have two cases: $f_3>1$ and $f_3=1$.
\par(i-a) If $f_3>1$, then $S_2=S_2'=S_3=\varnothing$.
\par (i-b) If $f_3=1$, then, from the equality $n=e_3f_3k$, we obtain $k=n/e_3$. 
Hence, 
\[
S_2=S_2'=\varnothing, \  \ \ S_3=\{\qqq_1,\qqq_2,\dots,\qqq_{\frac{n}{e_3}}\}.
\]
\par {\bf Case (ii)} $f_2=1$: Here, we have two cases: $e_2>1$ or $e_2=1$.
%\par(ii-a) We have two cases: $e_2>1$ or $e_2=1$. 
\par(ii-a) If $e_2>1$, 
then the prime $\ppp_i$ ramifies, for any $1 \leq i \leq m$. It follows from the 
relation $n=e_2f_2m$ that $m=n/e_2$. So, in this case, we have
\[
S_2=\varnothing, \ \ \ S_2^{\prime}=\{\ppp_1,\ppp_2,\dots,\ppp_{\frac{n}{e_2}}\}. 
\]
If $f_3>1$, then $S_3=\varnothing$. If $f_3=1$, then, as above,
$S_3=\{\qqq_1,\qqq_2,\dots,\qqq_{\frac{n}{e_3}}\}$.
\par(ii-b) Now, let $e_2=1$. Then 
\[
S_2=\{\ppp_1,\ppp_2,\dots,\ppp_n\}, \ \ \  S_2'=\varnothing. 
 \]
As in the previous case, if $f_3>1$, then $S_3=\varnothing$,
and if $f_3=1$, then
$S_3=\{\qqq_1,\qqq_2,\dots,\qqq_{\frac{n}{e_3}}\}$.

Now, the desired result follows from Theorem \ref{Main}.
\end{proof}

Although this theorem gives a description of $\SL_2(\mathcal{O}_K)^{\ab}$ when
$\OO_K$ is the ring of integers of a Galois extension $K/\q$, note that it depends 
on how the primes $2$ and $3$ are decomposed in $\mathcal{O}_K$. This task is not 
always easy, and sometimes we need to find these decompositions using the structure 
of $K$. 

\begin{exa}
Let $K=\q(\sqrt{-2}, \sqrt{3})$. It is easy to see that $K/\q$ is Galois. We know that 
$\displaystyle\OO_K=\z\Big[\frac{\sqrt{-6}+\sqrt{-2}}{2}\Big]$
(see \cite[Example 5.23]{keune2023}). Since $[K:\q] = 4$, $\OO_K^\times$ is infinite.
Let $\displaystyle\lambda := \frac{\sqrt{-6}+\sqrt{-2}}{2}$.
The factorizations of $2$ and $3$ in $\OO_K$ are
\[
2\OO_K = (2, \lambda + 1)^4, \ \ \ \ 3\OO_K = (3, \lambda - 1)^2 (3, \lambda + 1)^2.
\]
This shows that $e_2=4$, $f_2=1$ and $e_3=2$, $f_3=1$. Thus, by 
Theorem \ref{Galois}, we get
\[
\SL_2(\OO_K)^{\ab} \simeq (\z/2 \oplus \z/2) \oplus (\z/3)^2 \simeq \z/6\oplus \z/6.
\]
\end{exa}

We conclude by applying Theorem \ref{Galois} to cyclotomic fields.
Let $N$ be a positive integer and let $\zeta_N$ be the primitive $N$th root of unity.
Then the ring of algebraic integers of the cyclotomic field $K:=\q(\zeta_N)$ is
\[
\OO_K=\z[\zeta_N]
\]
(see \cite[Theorem 1.51]{keune2023}). Let $p \in \z$ be a prime and let
\[
p\mathcal{O}_K \simeq (\mathfrak{p}_1\mathfrak{p}_2\cdots \mathfrak{p}_{m_p})^{e_p}
\]
be the prime decomposition of $p\OO_K$.
Let $N=p^ns$, such that $p\nmid s$. Then $e_p=\phi(p^n)$ and the inertia degree 
$f_p$ is equal to the smallest positive integer such that ${N}/{p^n}=s \mid p^{f_p}-1$
(see \cite[Theorem 3.16]{keune2023}). 
%Observe that $\phi(N)=e_pf_pm_p$.

\begin{prp}\label{cyclotomic}
Let $N$ be a positive integer and let $\z[\zeta_N]$ be the ring of 
integers of the cyclotomic field $\q(\zeta_N)$. Then 
\[
\SL_2(\z[\zeta_N])^\ab \simeq \begin{cases}
\z/12     &  \text{if $N=1,2$}\\
\z/2 \oplus \z/2     &  \text{if $N=2^k$, $k\geq 2$,}\\
\z/3  & \text{if $N=2^k3^m$, $k\in \{0,1\}$, $m>0$.}\\
0 & \text{otherwise}
\end{cases}
\]
\end{prp}

\begin{proof}
If $N=1,2$, then $\z[\zeta_N])=\z$. It is known that $\SL_2(\z)^\ab\simeq \z/12$
(see \cite[page 35]{cohn1966}) or \cite[Example 3.4]{B-E2024}). If $N=3, 6$, then
\[
\z[\zeta_3]=\z[(1+\sqrt{-3})/2]=\OO_{-3}=\z[(1-\sqrt{-3})/2]=\z[\zeta_6],
\]
and, if $N=4$, then 
$\z[\zeta_4]=\z[\sqrt{-1}]=\OO_{-1}$. Cohn has shown that 
\[
\SL_2(\OO_{-3})^\ab\simeq \z/3, \ \ \text{and} \ \ \SL_2(\OO_{-1})^\ab=\z/2\oplus \z/2
\]
(see \cite[page 162]{cohn1968} or \cite[Proposition 4.7]{B-E2024}). 

Thus, we may assume that $N\neq 1, 2, 3, 4, 6$. Then $[\q(\zeta_N):\q] \geq 3$ and so, by 
Theorem~\ref{dir}, $\z[\zeta_N]^{\times}$ is infinite. Let us consider the prime decompositions
\begin{align*}
2\z[\zeta_N]=(\ppp_1\ppp_2\cdots \ppp_r)^{e_2},\ \ \ \ \ 3\z[\zeta_N]=(\qqq_1\qqq_2\cdots \qqq_s)^{e_3}.
\end{align*}
Let $f_2$ and $f_3$ be the inertia degrees of $\ppp_i$ and $\qqq_j$, respectively. 
    
First, let  $N=2^k$, for some $k>2$ ($N\neq 2,4$). Then, by the above discussion, $f_2=1$ and 
\[
e_2=\phi(2^k)=\phi(N)>1.
\]
On the other hand, since the largest integer $l \in \z_{\geq 0}$ such that 
$3^l\mid N=2^k$ is $l=0$, we must have $f_3>1$. Now, by Theorem~\ref{Galois}, we have
\[
\SL_2(\z[\zeta_{2^k}])^\ab\simeq \z/2 \oplus \z/2.
\]

If $N=3^m$, for $m>1$, then a similar argument to the previous case shows that $f_2>1$, 
$e_2=1$, $f_3=1$ and $e_3=3^m=N$. Thus 
\[
\SL_2(\z[\zeta_{3^m}])^\ab\simeq \z/3.
\]
Now, suppose that $N \neq 2^n, 3^m$. Hence $N=2^k3^l n$, where  $2\nmid n$,
$3\nmid n$, $k, l \in \z_{\geq 0}$ and $k, l \neq 0$ in case $n=1$. We know that
$3^ln=N/2^k$, so to have $N/2^k\mid 2^{f_2}-1$, we must have $f_2>1$.  Also, since 
\[
2^kn=N/3^l \mid 3^{f_3}-1,
\]
we have $f_3>1$, if $n >1$ or $k> 1$, in 
this case, by Theorem \ref{Galois}, $\SL_2(\z[\zeta_N])^{ab}=0$. Now, if neither 
$n>1$ nor $k>1$, i.e. $N=2\cdot 3^{l}$, we must have $f_3=1$. But $e_3=\phi(3^l)$, 
hence $\phi(N)/e_3=\phi(2)=1$. Therefore, by Theorem \ref{Galois},  
$\SL_2(\z[\zeta_{2\cdot3^l}])^{ab}=\z/3$, for any $l$. This concludes the proof.
\end{proof}

%%%%%%%%%%%%%%%%%%%%%%%%%%%%%%%%%%%%%%%%%%%%%%%%%%%%%%%%


\begin{thebibliography}{99}
%%%%%%%%%%%%%%%%%%%%%%%%%%%%%%%%%%%%%%%%%%%%%%%%%%%%%%%%
  
%\bibitem{an1998}
%Adem, A., Naffah, N. On the cohomology of $\SL_2(\z[1/p])$. In Geometry and cohomology in group 
%theory (Durham, 1994), volume 252 of London Math. Soc. Lecture Note Ser., pages 1--9. Cambridge Univ. P
%ress, Cambridge, 1998

\bibitem{am1969}
Atiyah,~M.~F., Macdonald,~I.~G. Introduction to Commutative Algebra, Addison–Wesley, Boston, (1969)

%\bibitem{bass1964}
%Bass, H. $K$-theory and stable algebra, Publications Math\'ematiques de 
%l’IHÉS {\bf 22} (1964), 5--60

\bibitem{bms1967}
Bass, H., Milnor, J., Serre, J.-P. Solution of the congruence subgroup problem for 
$\SL_n(n \geq 3)$ and ${\rm Sp}_{2n}(n \geq 2)$. Publications Math\'ematiques de 
l’IHÉS {\bf 33} (1967),  59--137

\bibitem{BS2013}
Boylan, H., Skoruppa, N.-P. Linear characters of $\SL_2$ over Dedekind domains. Journal of Algebra 
{\bf 373} (2013), 120--129

%\bibitem{brown1994}
%Brown, K. S. Cohomology of Groups. Graduate Texts in Mathematics, 87. Springer-Verlag, New York (1994)

%\bibitem{bd2019}
%Brown, S. C., Davis, C. T. The factorization of $2$ and $3$ in cyclic quartic fields. 
%Math. J. Okayama Univ. {\bf 61} (2019), 167--172

%\bibitem{ae2014}
%Bui, A. T., B., Ellis, G. The homology of $\SL_2(\z[1/m])$ for small $m$. Journal of 
%Algebra {\bf 408} (2014) 102--108

\bibitem{cohn1966}
Cohn, P. M. On the structure of the $\GL_2$ of a ring. Publications Math\'ematiques de 
l’IHÉS {\bf 30} (1966), 5--53

\bibitem{cohn1968}
Cohn, P. M. A presentation of $\SL_2$ for Euclidean imaginary quadratic number fields. 
Mathematika {\bf 15} (1968), no. 2, 156--163

%\bibitem{ep2005}
%Engler, A. J., Prestel, A. Valued Fields. Springer Monographs in Mathematics (SMM) (2005)

%\bibitem{funa1984}
%Funakura, T. On integral bases of pure quartic fields. Math. J. Okayama Univ. 26
%(1984), 27-41

%\bibitem{hsw1995}
%Huard, J. G., Spearman, B. K., Williams, K. S. Integral bases for quartic
%fields with quadratic subfields, J. Number Theory 51 (1995), 87-102

%\bibitem{h2022}
%Hutchinson, K. ${\rm GE}_2$-rings and a graph of unimodular rows. J. Pure Appl. Algebra 
%{\bf 226} (2022), no. 10, 107074

%\bibitem{h2016}
%Hutchinson, K. The second homology of $\SL_2$ of $S$-integers. Journal of Number Theory 
%{\bf 159} (2016), 223--272

%\bibitem{hmm2022}
%Hutchinson, K., Mirzaii, B.,  Mokari, F. Y. The homology of $\SL_2$ of discrete valuation rings. 
%Advances in Mathematics {\bf 402} (2022) 108313

%\bibitem{h2022}
%Hutchinson, K. ${\rm GE}_2$-rings and a graph of unimodular rows. J. Pure Appl. Algebra 
%{\bf 226} (2022), no. 10, 107074

%\bibitem{karpi1987}
%Karpilovsky, G. The Schur Multiplier. London Mathematical Society Monographs. Clarendon Press, Oxford (1987)

\bibitem{keune2023}
Keune, F. Number Fields. Radboud University Press. Version: 2023-01

%\bibitem{lam2005}
%Lam, T. Y. Introduction to Quadratic Forms over Fields. Graduate Studies in Mathematics. Vol. 67. 
%American Mathematical Society (2005)

%\bibitem{lam2006}
%Lam, T.Y. Serre’s Problem on Projective Modules. Springer Monograph in Mathematics, (2006)

%\bibitem{lang1994}
%Lang, S. Algebraic Number Theory. Second Edition, 
%Graduate Texts in Mathematics, 110, Springer, New York, (1994). 

\bibitem{lang2002}
Lang, S. Algebra. Revised Third Edition, 
Graduate Texts in Mathematics, Vol. 211, Springer-Verlag, New York, (2002)

\bibitem{liehl1981}
Liehl, B. On the group $\SL_2$ over orders of arithmetic type. J. Reine Angew. Math. {\bf 323} (1981)
153--171

%\bibitem{maclane}
%Mac Lane, S. Categories for the
%Working Mathematician. Second Edition, 
%Graduate Texts in Mathematics, 5, Springer-Verlag, New York, (1971)

%\bibitem{marcus1977}
%Marcus, D. A. Number Fields. Unviersitext, Springer-Verlag, New York-Heidelberg, (1977)

%\bibitem{ma1969}
%Matsumoto, H. Sur les sous-groupes arithmétiques des groupes semi-simples déployés, Ann. Sci. de
%l'É.N.S. 4e série, tome 2, n. 1 (1969), 1–62

%\bibitem{ma1989}
%Matsumura, H. Commutative Ring Theory. Cambridge Studies in Advanced Mathematics 8.
%Cambridge University Press (1989)

%\bibitem{men1967}
%Mennicke, J. On Ihara's modular group. Inventiones Math. {\bf 4} (1967), 202--228

%\bibitem{milne}
%Milne, J. S. Algebraic Number Theory. Version 3.08

%\bibitem{milnor}
%Milnor J. Introduction to Algebraic K-Theory, Ann. of Math. Stud., vol. 72, Princeton University
%Press, 1971

%\bibitem{B-E--2023}
%Mirzaii, B., Torres Pérez, E. A refined scissors congruence group and the third homology 
%of $\SL_2$.  Journal of Pure and Applied Algebra {\bf 228} (2024), 107615, 1--28

%\bibitem{bbt2025}
%Mirzaii, B., Ramos, B. R., Verissimo, T. The second integral homology of $\SL_2(\z[1/n])$. 
%Preprint: \texttt{https://arxiv.org/pdf/2503.12190}

\bibitem{B-E2024}
Mirzaii, B., Torres P\'erez, E. The abelianization of the elementary group of rank two. 
Proc. Edinburgh Math. Soc. {\bf 68} (2025), no. 2, 487--505
%1-19, doi:10.1017/S0013091524000877 

\bibitem{carl2024}
Nyberg-Brodda, C. The abelianization of $\SL_2(\z[\frac{1}{n}])$. Journal of Algebra 
{\bf 660} (2024), 614--618

%\bibitem{B-E-2024}
%Mirzaii, B., Torres Pérez, E.
%The third homology of projective special linear group of rank two. Available at 
%https://arxiv.org/abs/2401.04341.

%\bibitem{B-E--2024}
%Mirzaii, B., Torres Pérez, E. On the connections between the third homology of $\SL_2$ and $\PSL_2$. 
%Available at https://arxiv.org/abs/2402.08074

\bibitem{rahm2013}
Rahm, A. D. Higher torsion in the Abelianization of the full Bianchi groups. LMS J.
Comput. Math. {\bf 16} (2013), 344--365

\bibitem{rahm-2013}
Rahm, A. D. The homological torsion of $\PSL_2$ of the imaginary quadratic integers.
Trans. Amer. Math. Soc. {\bf 365} (2013), no. 3, 1603--1635

%\bibitem{rf2011}
%Rahm, A. D., Fuch, M. The integral homology of imaginary quadratic integers with nontrivial 
%class group. J. Pure Appl. Algebra {\bf 215} (2011), no. 6, 1443--1472

\bibitem{rv1998}
Ramakrishnan, D., Valenza, R. J. Fourier analysis on number fields, Vol. 186. Springer 
Science Business Media (1998)

%\bibitem{derek}
%Robinson, D. A Course in the Theory of Groups. Graduate Texts in Mathematics, 80. Springer-Verlag New 
%York City (1996)

%\bibitem{rosenberg}
%Rosenberg, J. Algebraic K-Theory and its applications. New York: Springer-Verlag, 1994. 
%(Graduate Texts in Mathematics, 147)

%\bibitem{rot2008}
%Rotman, J. J. An Introduction to Homological Algebra. Springer. Second edition. Universitext (2008)

%\bibitem{sah1989}
%Sah, C. Homology of classical Lie groups made discrete. III. J. Pure Appl. Algebra {\bf 56} (1989), no. 3, 
%269--312.

%\bibitem{samuelER}
%Samuel, P. About Euclidean Rings. Jounal of Algebra {\bf 19} (1971), 282--301

%\bibitem{serre1979}
%Serre J.-P. Arithmetic groups. In Homological Group Theory, London Math. Soc. Lect. Notes
%Series, No. 36, Cambridge University Press Cambridge, 1979, 105--135

\bibitem{serre1970}
Serre, J.-P. Le probl\`eme des groupes de congruence pour $\SL_2$, Ann. of Math. (2) 
{\bf 92} (1970), 489--527

%\bibitem{sw2006}
%Spearman, B. K., Williams, K. S. The prime ideal factorization of $2$ in pure quartic fields with 
%index $2$. Math. J. Okayama Univ. 48 (2006), 43–46

%\bibitem{serre1979}
%Serre, J.-P.:  Local Fields, Springer-Verlag, Berlin (1979)

%\bibitem{serre1980}
%Serre, J.-P. Trees, Springer-Verlag, Berlin  (1980) 

\bibitem{silv1982}
Silvester, J. R. On the $\GL_n$ of a semi-local ring. Algebraic $K$-Theory, Lecture 
Notes in Mathematics, Vol {\bf 966} (1982), 244--260

%\bibitem{stein2012}
%Stein, W. Algebraic Number Theory, a Computational Aproach. 2012.
%\bibitem{s1976}
%Suslin, A. On a theorem of Cohn. J. Soviet Math. 17 (1981), 1801–1803

%\bibitem{swan}
%Swan. R. G. Generators and relations for certain special linear groups. Bull. Amer. Math. Soc. {\bf 74} 
%(1968), 576--581

\bibitem{swan1971}
Swan. R. G. Generators and relations for certain special linear groups. 
Advances in Mathematics {\bf 6} (1971), no. 1, 1--77

%\bibitem{srinivas1996}
%Srinivas, V. Algebraic $K$-Theory. Second edition. Progress in Mathematics, 90. Birkh\"auser Boston, 1996

%\bibitem{suslin1991}
%Suslin, A. A. $K\sb 3$ of a field and the Bloch group. Proc. Steklov Inst. Math. {\bf 183} (1991), no. 4, 
%217--239 

\bibitem{vas1972}
Vaserstein, L. N. The group $\SL_2$ over Dedekind rings of arithmetic type. Math. USSR Sb. {\bf 18} 
(1972) 321--332

%\bibitem{weibel1994}
%Weibel, C. A. An Introduction to Homological Algebra. Cambridge Studies in Advanced Mathematics, 38. 
%Cambridge University Press, Cambridge, (1994)

%\bibitem{weibel2013}
%Weibel, C. A. The $K$-Book: An Introduction to Algebraic $K$-Theory. Graduate Studies in 
%Mathematics, vol. 145. American Mathematical Society, Providence (2013)

\bibitem{weiss1976}
Weiss, E. Algebraic Number Theory. McGraw-Hill, New York, (1963); reprinted, Chelsea, New
York (1976)

\end{thebibliography}
\end{document}